\documentclass[letter]{amsart}

\usepackage{amsmath}
\usepackage{amsfonts}
\usepackage{amssymb}
\usepackage[colorlinks=true, pdfstartview=FitV, linkcolor=blue, citecolor=blue, urlcolor=blue,pagebackref=false]{hyperref}
\usepackage[usenames]{color}

\numberwithin{equation}{section}
\newtheorem{theorem}{Theorem}[section]

\newtheorem{proposition}[theorem]{Proposition}
\newtheorem{lemma}[theorem]{Lemma}
\theoremstyle{definition}

\theoremstyle{remark}
\newtheorem{remark}{Remark}[section]

\DeclareMathOperator{\supp}{supp}

\definecolor{orange}{rgb}{1,0.5,0}
\definecolor{green}{rgb}{0.4,0.8,0.5}


\title[ ]
      {Increasing stability for optical tomography}
\author [ ]
       {Ru-Yu Lai}
\thanks{Department of Mathematics, University of Washington, Seattle, WA 98195, USA. Email: rylai@uw.edu}

\begin{document}

\maketitle
\setcounter{tocdepth}{1}

\begin{abstract}
We study the phenomenon of increasing stability in the diffuse optical tomography (DOT). It is well-known that the DOT inverse problem is exponentially ill-posed. In this paper, we show that the ill-posedness decreases when we increase the frequency.
\end{abstract}

\section{Introduction}
In this paper, we study the diffuse optical tomography \cite{BBMDKGZ}. DOT is an medical imaging modality in which tissue
is illuminated by near-infrared light from an array of sources, the
multiply-scattered light which emerges is observed with an array of
detectors, and then a model is used to infer the localized optical
properties of the illuminated tissue. The most important current
application of DOT are detecting tumors in the breast and imaging
the brain. The more blood supply of tumors compared to surrounding
tissue provides a target absorption inhomogeneity to image. A
similar idea allows us to image bleeding in the brain. It is
well-known that the DOT inverse problem is exponentially ill-posed
or unstable \cite{Ar}. The inverse problem is to recover coefficient
of absorption and diffusion in the tissue from the boundary
measurement of photon density. Let $\Omega\subset \mathbb{R}^n,
n\geq3$ be a bounded domain with smooth boundary. We assume that
the propagation of light through this medium can be modeled by the
diffusion approximation
\begin{align}\label{cond}
     \nabla\cdot \gamma\nabla u(x) +D(x)u(x)+k^2 u(x)=0\ \ \hbox{in $\Omega$},
\end{align}
where $u$ describes the photon density in the medium, $\gamma$ is the diffusion coefficient and $D$ the absorption coefficient.
We show that the stability increases as frequency $k$ is growing.

A short paper \cite{C} published by A. P. Calder\text{\'{o}}n in
1980 motivated many developments in inverse problems, in particular
in the construction of ``complex geometrical optics'' (CGO)
solutions of partial differential equations to solve inverse
problems. The problem that Calder\text{\'{o}}n considered was
whether one can determine the electrical conductivity of a medium by
making voltage and current measurements at the boundary of the
medium. This inverse method is known as \emph{Electrical Impedance
Tomography} (EIT). EIT arises not only in geophysical prospections
(See \cite{ZK}) but also in medical imaging (See \cite{Holder},
\cite{HIMS} and \cite{J}). We now describe more precisely the
mathematical problem. Let $\Omega\subset\mathbb{R}^n$ be a bounded
domain with smooth boundary. The electrical conductivity of $\Omega$
is represented by a bounded and positive function $\gamma(x)$. The
equation for the potential is given by
\begin{align*}
         \nabla\cdot \gamma\nabla u =0 & \ \ \hbox{in $\Omega$}.
\end{align*}
Given $f\in H^{1/2}(\partial\Omega)$ on the boundary, the potential
$u\in H^1(\Omega)$ solves the Dirichlet problem
\begin{align}\label{con}
     \left\{
       \begin{array}{rl}
         \nabla\cdot \gamma\nabla u =0 & \hbox{in $\Omega$} \\
         u=f & \hbox{on $\partial\Omega$.}
       \end{array}
     \right.
\end{align}
The Dirichlet-to-Neumann map, or voltage-to-current map, is given by
$$
     \Lambda_\gamma f=\gamma\partial_\nu u|_{\partial\Omega},
$$
where $\partial_\nu u=\nu\cdot\nabla u$ and $\nu$ is the unit outer normal to $\partial\Omega$. The well-known inverse problem is to recover the conductivity $\gamma$ from the boundary measurement $\Lambda_\gamma$.

The uniqueness issue for $C^2$ conductivities was first settled by
Sylvester and Uhlmann \cite{SU1}. Later, the regularity of
conductivity was relaxed to $3/2$ derivatives in some sense in
\cite{BT} and \cite{PPU}. Uniqueness for conductivities with
conormal singularities in $C^{1,\varepsilon}$ was shown in
\cite{GLU}. See \cite{U1} for the detailed development. Recently,
Haberman and Tataru \cite{HT} extended the uniqueness result to
$C^{1}$ conductivities or small in the $W^{1,\infty}$ norm.

A logarithmic stability estimate for conductivities was obtained by
Alessandrini \cite{A1}. Then Mandache \cite{M} showed that this
estimate is optimal. For a review of stability issues in EIT see
\cite{A2}. The logarithmic stability makes it difficult to design
reliable reconstruction algorithms in practice since small errors in
the data of the inverse problem result in large error in numerical
reconstruction of physical properties of the medium. However, it has
been observed numerically that the stability increases if one
increases the frequency in some cases. These papers (\cite{AI},
\cite{AI2}, \cite{HI}, \cite{I2007}, \cite{I}, \cite{NUW})
rigorously demonstrated the increasing stability phenomena in
different settings.

The outline of the paper is as follows. In section 2, we start with a more detailed description of the considered problem and main result. In section 3, we construct the complex geometrical optics solutions by following the idea of Haberman and Tataru \cite{HT}. Then we deduce a useful boundary integral estimate in section 4. The detailed proof of the main theorem is presented in section 5.

\section{Main results}
We assume that the diffusion coefficient $\gamma\in C^{1,\varepsilon}(\overline\Omega), 0<\varepsilon<1$ and the absorption coefficient $D\in L^\infty(\Omega)$. We consider the equation
\begin{align}\label{cond}
     \nabla\cdot \gamma\nabla u(x) +D(x)u(x)+k^2 u(x)=0\ \ \hbox{in $\Omega$}
\end{align}
with the Dirichlet boundary data
\begin{align}\label{cond11}
     u=g\ \ \hbox{on $\partial\Omega$}.
\end{align}

For real-valued $\gamma$, the Dirichlet problem (\ref{cond}),
(\ref{cond11}) might fail to exist and be unique, so that the
Dirichlet-to-Neumann map is not well-defined. Then one can consider
replace this map by the Cauchy data with naturally defined norm (see
\cite{NUW}). We will assume that $k$ is not a Dirichlet eigenvalue.
Then $\Lambda_\gamma$ is a continuous linear operator from
$H^{1/2}(\partial\Omega)$ into $H^{-1/2}(\partial\Omega)$. We denote
its operator norm by $\|\Lambda_\gamma\|_*$.

Let $v=\sqrt{\gamma} u$. The equation (\ref{cond}) can be transformed into this Helmholtz type equation
\begin{align}\label{helm}
     \Delta v(x) +Q(x) v(x)=0\ \ \hbox{in $\Omega$},
\end{align}
where $Q$ can be formally defined by $Q=q+
\left(k^2+D\right)\gamma^{-1}$ with
$q=-\frac{\Delta\sqrt{\gamma}}{\sqrt{\gamma}}$.

Now we have the main result.
\begin{theorem}\label{main}
     Let $\Omega\subset \mathbb{R}^n, n\geq3$ be a bounded domain with smooth boundary.
     Let $M$, $\delta$ and $\varepsilon$ be real constants such that $M>1$, $0<\delta<1$ and $0<\varepsilon<1$. Suppose that $\|\gamma_j\|_{C^{1,\varepsilon}(\overline\Omega)}\leq M$ with $\gamma_j(x)>1/M$
     for all $x\in\Omega$ and $\|D_j\|_{L^\infty(\Omega)}\leq M$
     and
     $$
     A=\max\Big\{\|\Lambda_{\gamma_1}-\Lambda_{\gamma_2}\|^2_*, \|\Lambda_{\gamma_1}-\Lambda_{\gamma_2}\|^{2\delta}_*\Big\}.
     $$
     Assume that $-\log A\geq 1$.
     There exists a constant C such that if $k\geq 1$, then we have the following stability estimates:
     \begin{align}\label{maininequ11}
         &\|\left(q_1+\left(k^2+D_1\right)\gamma_1^{-1}\right)-\left(q_2+\left(k^2+D_2\right)\gamma_2^{-1}\right)\|_{H^{-s}(\Omega)}  \notag\\
         &\leq  C e^{Ck^\alpha}A^{1/2}
         + C\max\left\{k^{\frac{-\alpha\varepsilon}{1+\varepsilon}},k^{4-\alpha},k^2\left(k^\alpha+\log \frac{1}{A}\right)^{1-s} \right\},\ \ \alpha>4
     \end{align}
     and
     \begin{align}\label{maininequ}
         \|\gamma_1^{-1}-\gamma_2^{-1}\|_{H^{-s}(\Omega)}\leq  C e^{Ck^\alpha}A^{1/2}
         + C\max\left\{ k^{-2},k^{2-\alpha} \right\},\ \ \alpha>2
     \end{align}
with $2s>n+3$ and $C>0$ depends on $n,s,\varepsilon,\Omega$ and $M$.

\end{theorem}

\begin{remark}
The estimate (\ref{maininequ11}) consists of two parts-Lipschitz and logarithmic estimates. The logarithmic part decreases and the Lipschitz part becomes dominant when the frequency k is growing. In other words, the stability is increasing if $k$ is large.

\end{remark}

\section{Complex geometrical optics solutions}
We construct CGO solutions to (\ref{cond}) with less regular
conductivities in this section. To deal with less regular
conductivities $\gamma$, we adopt Bourgain-type spaces introduced by
Haberman and Tataru in \cite{HT}.

Let $\zeta\in \mathbb{C}^n, n\geq 3$ and let $p_\zeta$ denote the polynomial
$$
           p_\zeta(\xi)=-|\xi|^2+2i\zeta\cdot\xi,
$$
which is the symbol of the operator $\Delta_\zeta=\Delta+2\zeta\cdot \nabla.$
For any $b\in \mathbb{R}$, we define spaces $\dot{X}^b_\zeta$ by the norm
$$
     \|u\|_{\dot{X}^b_\zeta}=\||p_\zeta(\xi)|^b \hat{u}(\xi)\|_{L^2}.
$$
We will only use the cases where $b\in \{1/2,-1/2\}$. Note that $\dot{X}^{-1/2}_\zeta$ can be identified as the dual space of  $\dot{X}^{1/2}_\zeta$. One feature of these spaces is that the operator $\Delta_\zeta^{-1}$ is a bounded linear operator from  $\dot{X}^{-1/2}_\zeta$ to $\dot{X}^{1/2}_\zeta$ with norm
$$
     \|\Delta_\zeta^{-1}\|_{\mathcal{L}\left(\dot{X}^{-1/2}_\zeta\rightarrow \dot{X}^{1/2}_\zeta\right)}=1.
$$

Assume that the conductivity $\gamma$ is in the space $C^{1,\varepsilon}(\overline\Omega), 0<\varepsilon<1$ with  $\|\gamma_j\|_{C^{1,\varepsilon}(\overline\Omega)}\leq M$ and $\gamma_j(x)>1/M$. By Proposition 2.4 in \cite{CGR}, there exist a ball $B=B(0,R)$ with radius $R>0$ and $\sigma$ in  $ C^{1,\varepsilon}(\mathbb{R}^n)$ such that $\overline \Omega\subset B$, $\gamma=\sigma|_{\overline\Omega}$ and $\supp(\sigma-1) \subset \overline B$.
Now we use the same notation to express the conductivity $\gamma\in C^{1,\varepsilon}(\mathbb{R}^n)$ which satisfies $\|\gamma_j\|_{C^{1,\varepsilon}(\mathbb{R}^n)}\leq M$ and $\gamma_j(x)>1/(2M)$ for all $x\in \mathbb{R}^n$.
Moreover, we can also extend the coefficient $k^2+D$ by zero to $\mathbb{R}^n\setminus \Omega$.
Hence the equation (\ref{helm}) can be extended to $\mathbb{R}^n$ in the following sense:
\begin{align}\label{helm1}
     \Delta v +Q(x) v(x)=0\ \ \hbox{in $\mathbb{R}^n$},
\end{align}
where  $Q=q+ \left(k^2+D\right)\gamma^{-1}$ with
$q=-\frac{\Delta\sqrt{\gamma}}{\sqrt{\gamma}}$. Suppose that the CGO
solutions of (\ref{helm1}) has the form
$$
     v=e^{\zeta\cdot x}(1+\psi(x))\ \ \hbox{in $\mathbb{R}^n$}
$$
with $\zeta\in \mathbb{C}^n$ satisfies $\zeta\cdot\zeta=0$.
We define
$$
     \langle Q| \phi_1\rangle :=\int_{\mathbb{R}^n} \nabla \gamma^{1/2}\cdot \nabla\left(\gamma^{-1/2}\phi_1\right) dx+\int_{\Omega} \left(k^2+D\right)\gamma^{-1}\phi_1 dx
$$
and
$$
     \langle m_Q \phi_1| \phi_2\rangle :=\int_{\mathbb{R}^n} \nabla \gamma^{1/2}\cdot \nabla\left(\gamma^{-1/2}\phi_1\phi_2\right) dx+\int_{\Omega} \left(k^2+D\right)\gamma^{-1}\phi_1\phi_2 dx
$$
for any $\phi_1,\phi_2\in C^\infty_0(\mathbb{R}^n)$.
In order to construct CGO solutions for (\ref{helm1}), the remainder $\psi$
need to satisfy
$$
     \Delta_\zeta \psi+m_Q \psi=-Q\ \ \hbox{in $\mathbb{R}^n$}
$$
with $\zeta\cdot\zeta=0$. The existence of $\psi$ is proved in the following theorem.
\begin{theorem}\label{CGOestimate}
     Suppose that $k\geq 1$. Let $\zeta\in \mathbb{C}^n$ satisfy $\zeta\cdot \zeta=0$. Then there exists a positive constant $C_*$ depending only on $n,\Omega$ and $M$ such that if
         $$
              |\zeta|\geq C_* k^2,
         $$
     then there exists a unique solution $\psi\in \dot X^{1/2}_\zeta$ to the equation
\begin{align}\label{Q}
     \Delta_\zeta \psi+m_Q \psi=-Q
\end{align}
    satisfying the estimate
    $$
         \|\psi\|_{\dot X^{1/2}_\zeta}\leq C \|Q\|_{\dot X^{-1/2}_\zeta},
    $$
where $C$ is independent of $k$.
\end{theorem}
\begin{proof}
      By using the Neumann series argument (see \cite{SU1}), we can show the existence of $\psi\in \dot{X}^{1/2}_\zeta$ which satisfies
     $$
      \|\psi\|_{\dot{X}^{1/2}_\zeta}\leq    \|(I+\Delta_\zeta^{-1}\left(m_Q\right))^{-1}\|_{\mathcal{L}\left(\dot{X}^{1/2}_\zeta\rightarrow \dot{X}^{1/2}_\zeta\right)}   \|\Delta_\zeta^{-1}(-Q)\|_{\dot{X}^{1/2}_\zeta}
     $$
     for $|\zeta|$ large enough and
     \begin{align}\label{psi}
     \|m_Q\|_{\mathcal{L}\left(\dot X^{1/2}_\zeta\rightarrow  \dot X^{-1/2}_\zeta \right) }< 1.
     \end{align}
     To prove (\ref{psi}),
     let $\phi_1, \phi_2\in \dot X^{1/2}_\zeta$. To estimate
     \begin{align*}
     \langle m_Q\phi_1|\phi_2\rangle &=\int_{\mathbb{R}^n} \nabla \gamma^{1/2}\cdot \nabla\left(\gamma^{-1/2}\phi_1\phi_2\right) dx
     +\int_{\Omega}\left(k^2+D\right)\gamma^{-1}\phi_1\phi_2 dx\\
     &=\int_{\mathbb{R}^n} \nabla \gamma^{1/2}\cdot \nabla\gamma^{-1/2} \phi_1\phi_2 dx+\int_{\mathbb{R}^n} \gamma^{-1/2}\nabla \gamma^{1/2}\cdot \nabla(\phi_1\phi_2) dx\\
     &\quad+\int_{\Omega} \left(k^2+D\right)\gamma^{-1}\phi_1\phi_2 dx\\
     &=: I_1+I_2+I_3.
     \end{align*}
     The terms $I_1$ and $I_3$ are bounded by
      \begin{align*}
          |I_1|\leq  \frac{C}{|\zeta|} \|\phi_1\|_{\dot X^{1/2}_\zeta}\|\phi_2\|_{\dot X^{1/2}_\zeta}
     \end{align*}
     and
     \begin{align*}
          |I_3|\leq C \frac{k^2}{|\zeta|} \|\phi_1\|_{\dot X^{1/2}_\zeta}\|\phi_2\|_{\dot X^{1/2}_\zeta}
     \end{align*}
     by Corollary 2.1 in \cite{HT}. Here $C$ is independent of $k$.

     Let $\Psi$ be a smooth function in $\mathbb{R}^n$ with support in the unit ball and $\int \Psi dx=1$. Define  $\Psi_h=h^{-n}\Psi(x/h)$ with $h>0$. We consider
     \begin{align*}
     |I_2|& \leq \left|\int_{\mathbb{R}^n} \Psi_h *\left(\gamma^{-1/2}\nabla \gamma^{1/2}\right)\cdot \nabla(\phi_1\phi_2) dx\right|\\
     &\quad+\left|\int_{\mathbb{R}^n} \left(\Psi_h *(\gamma^{-1/2}\nabla \gamma^{1/2})-\gamma^{-1/2}\nabla \gamma^{1/2}\right)\cdot \nabla(\phi_1\phi_2) dx\right|.
     \end{align*}
     By Lemma 2.3 in \cite{HT}, we get
     \begin{align*}
     \left|\int_{\mathbb{R}^n} \Psi_h *\left(\gamma^{-1/2}\nabla \gamma^{1/2}\right)\cdot \nabla(\phi_1\phi_2) dx\right|
     \leq  C|\zeta|^{-1} h^{-1}  \|\phi_1\|_{\dot X^{1/2}_\zeta}\|\phi_2\|_{\dot X^{1/2}_\zeta}
     \end{align*}
     and
     \begin{align*}
     &\left|\int_{\mathbb{R}^n} \left(\Psi_h *(\gamma^{-1/2}\nabla \gamma^{1/2})-\gamma^{-1/2}\nabla \gamma^{1/2}\right)\cdot \nabla(\phi_1\phi_2) dx\right|\\
     &\leq \|\Psi_h *(\gamma^{-1/2}\nabla \gamma^{1/2})-\gamma^{-1/2}\nabla \gamma^{1/2}\|_{L^\infty(\mathbb{R}^n)}  \|\phi_1\|_{\dot X^{1/2}_\zeta}\|\phi_2\|_{\dot X^{1/2}_\zeta}\\
     &\leq h^\varepsilon \|\phi_1\|_{\dot X^{1/2}_\zeta}\|\phi_2\|_{\dot X^{1/2}_\zeta},
     \end{align*}
     where the last inequality is based on the fact that $ \gamma\in C^{1+\varepsilon}(\mathbb{R}^n)$.
     Let $h=|\zeta|^{-1/(1+\varepsilon)}$, then
     \begin{align}\label{mQ}
          \|m_Q\|_{\mathcal{L}\left(\dot X^{1/2}_\zeta\rightarrow  \dot X^{-1/2}_\zeta  \right)}
          \leq
          C\left(|\zeta|^{-1}+|\zeta|^{-\varepsilon/(1+\varepsilon)}+k^2|\zeta|^{-1}\right),
     \end{align}
     where $C$ depends on $n, \Omega$ and $M$. Let $|\zeta|\geq C_* k^2$ with $C_*>0$ sufficiently large, we have
     $$
     \|m_Q\|_{\mathcal{L}\left(\dot X^{1/2}_\zeta\rightarrow  \dot X^{-1/2}_\zeta  \right)}<1.
     $$
     Hence the solution of (\ref{Q}) exists.
\end{proof}

In order to show that the remainder $\psi$ goes to zero as $|\zeta|\rightarrow \infty$, we need to have $\|Q\|_{\dot X^{-1/2}_\zeta}\rightarrow 0$ as $|\zeta|\rightarrow\infty$.
If $\gamma$ is smooth enough, one can prove
\begin{align}\label{q12}
\|Q\|_{\dot X^{-1/2}_\zeta}\leq C|\zeta|^{-\varepsilon}
\end{align}
for $\varepsilon\in (0,1/2]$ and $C$ is independent of $\zeta$.
However, (\ref{q12}) fails for less regular conductivities.
Hence, Haberman and Tataru established the estimate $\|Q\|_{\dot X^{-1/2}_\zeta}$ decays on average for some choices of $\zeta$.

Let $r\geq 0$ and $\eta\in S^{n-1}$. We set
$$
      \zeta=\tau\eta_1+i(\beta-r\eta/2),
$$
where $\tau>0$, $\beta\in \mathbb{R}^n$ and $\eta_1\in S^{n-1}$ satisfy $\eta_1\cdot \beta=\eta_1\cdot \eta=\eta \cdot \beta=0$.
The vector $\zeta$ is chosen so that $\zeta\cdot \zeta=0$ and $|\zeta|^2=2\tau^2$. Then we have the following Lemma.
\begin{lemma}\label{intQ}
     Suppose that $k\geq 1$. Then
     \begin{align*}
     \frac{1}{\lambda}\int_{S^{n-1}} \int^{2\lambda}_\lambda \|Q\|^2_{\dot X^{-1/2}_\zeta} d\tau d\eta_1
           \leq C \left( \left(1+\langle r\rangle^2\lambda^{-1}\right)\lambda^{-\varepsilon/(1+\varepsilon)} +k^4\lambda^{-1}  \right)
     \end{align*}
     if $\lambda$ is sufficiently large.
\end{lemma}
\begin{proof}
Let $\varphi\in C^\infty_0(\mathbb{R}^n)$ and $\varphi=1$ on $B$.
By the definition of $Q$, for $\phi\in \dot X^{1/2}_\zeta$, we have
\begin{align*}
     |\langle Q| \phi\rangle| &=\left|\int_{\mathbb{R}^n} \nabla \gamma^{1/2}\cdot \nabla\left(\gamma^{-1/2}\phi\right) dx+\int_{\Omega} \left(k^2+D\right) \gamma^{-1}\phi dx\right|\\
     &= \left|\frac{1}{4}\int_{\mathbb{R}^n} |\nabla\log\gamma|^2\phi dx+\frac{1}{2}\int_{\mathbb{R}^n}\nabla\log \gamma\cdot \nabla\phi dx+ \int_{\Omega}\left(k^2+D\right) \gamma^{-1} \phi dx\right|\\
     &\leq  C \left(|\zeta|^{-1/2}\sum^n_{j=1}\|\partial_j\log\gamma\|^2_{L^\infty(\mathbb{R}^n)} +\| \varphi\Delta\log \gamma\|_{\dot X^{-1/2}_\zeta} +k^2 |\zeta|^{-1/2}\right)\|\phi\|_{\dot X^{1/2}_\zeta}.
\end{align*}
Note that the last inequality is obtained by using Lemma 2.2 in \cite{HT}.
Thus,
\begin{align}\label{lemmaq}
     \|Q\|_{\dot X^{-1/2}_\zeta}\leq  C\left(|\zeta|^{-1/2} +\| \varphi\Delta\log \gamma\|_{\dot X^{-1/2}_\zeta} +k^2 |\zeta|^{-1/2}\right)
\end{align}
with $C $ depends on $n,\Omega$ and $M$.
To estimate the second term of (\ref{lemmaq}), we consider
\begin{align*}
     &\frac{1}{\lambda}\int_{S^{n-1}} \int^{2\lambda}_\lambda\|\varphi \Delta\log\gamma\|^2_{\dot X^{-1/2}_\zeta} d\tau d\eta_1 \\
     &\leq C\frac{1}{\lambda}\int_{S^{n-1}} \int^{2\lambda}_\lambda \| \varphi \nabla\cdot( \Psi_h*\nabla\log\gamma)\|^2_{\dot X^{-1/2}_\zeta}   d\tau d\eta_1\\
     &\quad+  C\frac{1}{\lambda}\int_{S^{n-1}} \int^{2\lambda}_\lambda   \| \varphi \nabla\cdot( \Psi_h*\nabla\log\gamma-\nabla\log\gamma)\|^2_{\dot X^{-1/2}_\zeta}  d\tau d\eta_1.
\end{align*}
We deduce
\begin{align*}
&\frac{1}{\lambda}\int_{S^{n-1}} \int^{2\lambda}_\lambda \| \varphi \nabla\cdot( \Psi_h*\nabla\log\gamma)\|^2_{\dot X^{-1/2}_\zeta}   d\tau d\eta_1\\
&\leq C \frac{1}{\lambda}\|  \nabla\cdot( \Psi_h*\nabla\log\gamma)\|^2_{L^2(\mathbb{R}^n)}
\leq C\frac{1}{h^2\lambda} \sum^n_{j=1}\|\partial_j\log\gamma\|^2_{L^\infty(\mathbb{R}^n)}
\end{align*}
and
\begin{align*}
& C\frac{1}{\lambda}\int_{S^{n-1}} \int^{2\lambda}_\lambda   \| \varphi \nabla\cdot( \Psi_h*\nabla\log\gamma-\nabla\log\gamma)\|^2_{\dot X^{-1/2}_\zeta}  d\tau d\eta_1\\
 &\leq C  \left(1+\langle r\rangle^2\lambda^{-1}\right) \| \Psi_h*\nabla\log\gamma-\nabla\log\gamma\|^2_{L^2(\mathbb{R}^n)}\leq C  \left(1+\langle r\rangle^2\lambda^{-1}\right) h^{2\varepsilon}.
\end{align*}
from Lemma 3.1 in \cite{HT}. Note that $\langle r\rangle=(1+r^2)^{1/2}$.
Summing up, we have
\begin{align*}
         \frac{1}{\lambda}\int_{S^{n-1}} \int^{2\lambda}_\lambda \|Q\|^2_{\dot X^{-1/2}_\zeta} d\tau d\eta_1
    \leq  C\left(\lambda^{-1} +h^{-2}\lambda^{-1}
    + \left(1+\langle r\rangle^2\lambda^{-1}\right) h^{2\varepsilon}+ k^4\lambda^{-1}\right).
\end{align*}
Taking $h=\lambda^{-1/(2+2\varepsilon)}$, since $0<\varepsilon<1$ and $\lambda \geq 1$, we complete the proof.

\end{proof}

\section{A boundary integral estimate}
In this section we derive a useful boundary integral estimate which is similar to Lemma 2.1 in \cite{CGR}.
\begin{lemma}
     Let $\gamma_j\in C^{1,\varepsilon}(\mathbb{R}^n)$ be two given functions with positive lower bound and $\supp(\gamma_j-1)\subset \overline B$ for $j=1,2$. Suppose that $\supp(D_1-D_2)\subset \Omega$. Then for any $u_j\in H^1_{loc}(\mathbb{R}^n)$ weak solution of $\nabla\cdot (\gamma_j\nabla u_j)+D_j u_j+k^2u_j=0$ in $\mathbb{R}^n$, one has
     \begin{align}\label{int}
         & \Big|\int_{\mathbb{R}^n} \nabla \gamma_2^{1/2}\cdot \nabla \left(\gamma_2^{-1/2}v_1v_2\right)dx-  \int_{\mathbb{R}^n} \nabla \gamma_1^{1/2}\cdot \nabla \left(\gamma_1^{-1/2}v_1v_2\right)dx   \notag\\
         &+\int_{B}\left(\left(k^2+D_2\right)\gamma_2^{-1}-\left(k^2+D_1\right)\gamma_1^{-1}\right)v_1v_2 dx  \Big|   \notag\\
          &\leq  \left(\|\Lambda_{\gamma_1}-\Lambda_{\gamma_2}\|_*+ \|\gamma_1- \gamma_2\|_{L^\infty(B\setminus\Omega)}\right)\|u_1\|_{H^1(B)}\|u_2\|_{H^1(B)}
     \end{align}
with $v_j=\gamma_j^{1/2}u_j$.
\end{lemma}

\begin{proof}
Since $\gamma_1=\gamma_2$ in $\mathbb{R}^n \setminus B$, applying integration by parts twice, we obtain that
\begin{align*}
     \langle  \Lambda_{\gamma_j}u_j, u_k \rangle
     &=  \int_B \gamma_j\nabla u_j\cdot \nabla u_k dx-\int_{B\setminus\Omega} \gamma_j\nabla u_j\cdot \nabla u_k dx-\int_\Omega \left(k^2+D_j\right)u_j u_k dx\\
     &=  \int_B \gamma_j \nabla u_j\cdot \nabla \left(\gamma_j^{-1/2} v_k\right) dx +\int_B\gamma_j\nabla u_j\cdot\nabla\left(\left(\gamma_k^{-1/2}-\gamma_j^{-1/2}\right)v_k\right)dx \\
     &\quad-\int_{B\setminus\Omega} \gamma_j\nabla u_j\cdot\nabla u_k dx-\int_\Omega \left(k^2+D_j\right)u_j u_k dx\\
     &=  \int_B \gamma_j \nabla u_j\cdot \nabla \left(\gamma_j^{-1/2} v_k\right) dx +\int_B \left(k^2+D_j\right)u_j\left(\gamma_k^{-1/2}-\gamma_j^{-1/2}\right)v_k dx \\
     &\quad-\int_{B\setminus\Omega} \gamma_j\nabla u_j\cdot\nabla u_k dx-\int_\Omega \left(k^2+D_j\right)u_j u_k dx,
\end{align*}
where $v_k=\gamma_k^{1/2}u_k$.
Since $\langle \Lambda_{\gamma_j}f, g\rangle=\langle \Lambda_{\gamma_j} g, f\rangle$, we have
\begin{align*}
     &\langle  \left(\Lambda_{\gamma_1}-\Lambda_{\gamma_2}\right)u_1, u_2 \rangle+ \int_{B\setminus\Omega} (\gamma_1-\gamma_2)\nabla u_1\cdot\nabla u_2 dx\\
     &=  \int_B \gamma_1 \nabla \left(\gamma_1^{-1/2}v_1\right)\cdot \nabla \left(\gamma_1^{-1/2} v_2\right) dx -\int_B \gamma_2 \nabla \left(\gamma_2^{-1/2}v_2\right)\cdot \nabla \left(\gamma_2^{-1/2} v_1\right) dx \\
     &\quad+ \int_B \left(\left(k^2+D_2\right)\gamma_2^{-1}-\left(k^2+D_1\right)\gamma_1^{-1}\right)v_1v_2 dx+\int_{B\setminus\Omega}\left(D_1-D_2\right)u_1u_2dx.
\end{align*}
     Applying $\gamma_1=\gamma_2=1$ outside $\overline B$, $D_1=D_2$ outside $\Omega$ and
\begin{align*}
      \int_B \gamma_1 \nabla \left(\gamma_1^{-1/2}v_1\right)\cdot \nabla \left(\gamma_1^{-1/2} v_2\right) dx
      =\int_B\nabla v_1\cdot \nabla v_2 dx-\int_B \nabla\gamma_1^{1/2}\cdot\nabla\left(\gamma_1^{-1/2}v_1v_2\right)dx,
\end{align*}
  then
  \begin{align*}
     &\langle  (\Lambda_{\gamma_1}-\Lambda_{\gamma_2})u_1, u_2 \rangle+ \int_{B\setminus\Omega} (\gamma_1-\gamma_2)\nabla u_1\cdot\nabla u_2 dx\\
     &= \int_{\mathbb{R}^n} \nabla \gamma_2^{1/2}\cdot \nabla \left(\gamma_2^{-1/2}v_1 v_2\right) dx- \int_{\mathbb{R}^n} \nabla \gamma_1^{1/2}\cdot \nabla \left(\gamma_1^{-1/2}v_1 v_2\right) dx  \\
     &\quad+\int_B \left(\left(k^2+D_2\right)\gamma_2^{-1}-\left(k^2+D_1\right)\gamma_1^{-1}\right)v_1v_2 dx,
\end{align*}
which implies (\ref{int}).

\end{proof}

\section{Proof of the main theorem}
This section is devoted to the proof of Theorem \ref{main}.
Let $r\geq 0$ and $\eta\in S^{n-1}$. We pick two vectors $\beta \in \mathbb{R}^n$, $\eta_1\in S^{n-1}$ such that
$     \beta\cdot\eta_1=\beta\cdot\eta=\eta\cdot\eta_1=0$, $|\tau|^2=|\beta|^2+r^2/4.$
Denote
$$
     \zeta_1=\tau\eta_1+i(\beta-r\eta/2),\ \ \zeta_2=-\tau\eta_1-i(\beta+r\eta/2).
$$
The two vectors $\zeta_l$ are chosen so that
$\zeta_l\cdot\zeta_l=0,$ $\zeta_1+\zeta_2=-ir\eta,$ $ |\zeta_l|^2=2\tau^2$ for $l=1,2.$
By Theorem \ref{CGOestimate}, if $|\zeta_l|\geq C_*k^2$ ($C_*$ is the constant given in Theorem \ref{CGOestimate}), we can construct CGO solutions $v_l$ to the equation (\ref{helm1}) with $Q=Q_l$ of the form
$$
     v_l(x)=e^{\zeta_l\cdot x}(1+\psi_l(x))
$$
with $\|\psi_l\|_{\dot X^{1/2}_{\zeta_l}}\leq C \|Q_l\|_{\dot X^{-1/2}_{\zeta_l}}$ for $l=1,2$ .
We plug CGO solutions $v_l$ into the left hand side of (\ref{int}), then we have
\begin{align}\label{l1}
     &\int_{\mathbb{R}^n} \nabla \gamma_2^{1/2}\cdot \nabla \left(\gamma_2^{-1/2}v_1v_2\right)dx-  \int_{\mathbb{R}^n} \nabla \gamma_1^{1/2}\cdot \nabla \left(\gamma_1^{-1/2}v_1v_2\right)dx   \notag\\
         &+\int_{\Omega} \left(\left(k^2+D_2\right)\gamma_2^{-1}-\left(k^2+D_1\right)\gamma_1^{-1}\right)v_1v_2 dx  \notag\\
        & =\langle  Q_2-Q_1| e^{-ir\eta\cdot x} \rangle+\langle  Q_2-Q_1| e^{-ir\eta\cdot x} (\psi_1+\psi_2)\rangle \notag\\
         &\quad-\langle  m_{Q_1} \psi_1| e^{-ir\eta\cdot x}\psi_2 \rangle+\langle  m_{Q_2} \psi_2| e^{-ir\eta\cdot x}\psi_1 \rangle.
\end{align}
Let $\varphi\in C^\infty_0(\mathbb{R}^n)$ and $\varphi=1$ on $B$. First we estimate
\begin{align}\label{l2}
           |\langle  Q_j| e^{-ir\eta\cdot x} (\psi_1+\psi_2)\rangle|
         &= |\langle  Q_j| \varphi e^{-ir\eta\cdot x} (\psi_1+\psi_2)\rangle| \notag\\
         &\leq   \|Q_j\|_{ X^{-1/2}_{\zeta_j}}\left(\| \varphi e^{-ir\eta\cdot x} \psi_1\|_{ X^{1/2}_{\zeta_1}}+\| \varphi e^{-ir\eta\cdot x} \psi_2\|_{X^{1/2}_{\zeta_2}}  \right) \notag\\
         &\leq C \langle r\rangle^{1/2}   \|Q_j\|_{\dot X^{-1/2}_{\zeta_j}}\left(\| \psi_1\|_{\dot X^{1/2}_{\zeta_1}}+\| \psi_2\|_{\dot X^{1/2}_{\zeta_2}}  \right)
\end{align}
by using Lemma 2.2 in \cite{HT} for $j=1,2$. Since $\gamma_j\in C^{1,\varepsilon}(\mathbb{R}^n)$, by the operator norm (\ref{mQ}) of $m_Q$, we obtain
\begin{align}\label{l3}
               |\langle  m_{Q_1} \psi_1| e^{-ir\eta\cdot x}\psi_2 \rangle|
         &\leq  C\left(\tau^{-1}+\tau^{-\varepsilon/(1+\varepsilon)} +k^2\tau^{-1}\right)\langle r\rangle^{1/2}\| \psi_1\|_{\dot X^{1/2}_{\zeta_1}}\| \psi_2\|_{\dot X^{1/2}_{\zeta_2}}
\end{align}
with $\tau>1$.
Combining (\ref{int}), (\ref{l1}), (\ref{l2}) and (\ref{l3}), it follows that
\begin{align}\label{Q12}
              |\langle  Q_2-Q_1| e^{-ir\eta\cdot x} \rangle|
        &\leq  C \langle r\rangle^{1/2} \sum^2_{l,j=1}  \|Q_j\|_{\dot X^{-1/2}_{\zeta_j}}\| \psi_l\|_{\dot X^{1/2}_{\zeta_l}} \notag\\
        &\quad+ C\left(\tau^{-1}+\tau^{-\varepsilon/(1+\varepsilon)} +k^2\tau^{-1}\right) \langle r\rangle^{1/2}\| \psi_1\|_{\dot X^{1/2}_{\zeta_1}}\| \psi_2\|_{\dot X^{1/2}_{\zeta_2}} \notag\\
        &\quad+\left(\|\Lambda_{\gamma_1}-\Lambda_{\gamma_2}\|_*+ \|\gamma_1- \gamma_2\|_{L^\infty(B\setminus\Omega)}\right)\|u_1\|_{H^1(B)}\|u_2\|_{H^1(B)}.
\end{align}
The estimate (2.6) in \cite{CGR} gives
$$
 \|\gamma_1- \gamma_2\|_{L^\infty(B\setminus\Omega)}\leq \left(\sum_{|\alpha|\leq 1} \|\partial^\alpha\gamma_1- \partial^\alpha\gamma_2\|_{L^\infty(\partial\Omega)}\right).
$$
Moreover, the same arguments in \cite{A} give
$$
\|\gamma_1-\gamma_2\|_{L^\infty(\partial\Omega)}\leq C\|\Lambda_{\gamma_1}-\Lambda_{\gamma_2}\|_*
$$
and
$$
 \sum_{|\alpha|= 1}\|\partial^\alpha \gamma_1-\partial^\alpha \gamma_2\|_{L^\infty(\partial\Omega)}\leq C\|\Lambda_{\gamma_1}-\Lambda_{\gamma_2}\|^\delta_*,\ 0<\delta<1,
$$
where $C$ depends on $n,\Omega,\varepsilon$ and $M$.
Thus
\begin{align}
\|\gamma_1- \gamma_2\|_{L^\infty(B\setminus\Omega)}\leq A^{1/2}
\end{align}
with
$$A=\max\left\{\|\Lambda_{\gamma_1}-\Lambda_{\gamma_2}\|^2_*, \|\Lambda_{\gamma_1}-\Lambda_{\gamma_2}\|^{2\delta}_*\right\}.$$
Recall that $B=B(0,R)$ is a ball with radius $R>0$. To show that
\begin{align}\label{uH1}
\|u_l\|^2_{H^1(B)}\leq  C e^{R|\zeta_l|},
\end{align}
where $|\zeta_l|\geq 1$ and $C$ depends on $n, \Omega$ and $M$.
First we compute
    \begin{align*}
         \|u_l\|^2_{L^2(B)} &=\|\gamma_l^{-1/2}e^{\zeta_l\cdot x}(1+\psi_l)\|_{L^2(B)}^2\\
         &\leq  Ce^{R|\zeta_l|}\left(1+\|\psi_l\|^2_{L^2(B)}\right)\\
         &\leq  Ce^{R|\zeta_l|}\left(1+|\zeta_l|^{-1}\|\psi_l\|^2_{\dot{X}^{1/2}_{\zeta_l}}\right).
    \end{align*}
Using the fact that $|\xi|^2/2\leq |p_\zeta(\xi)|\leq 3|\xi|^2/2$ when $4|\zeta|\leq |\xi|$. We can deduce
\begin{align*}
     \|\nabla \psi_l\|^2_{L^2(B)}
     &\leq   \|\nabla (\phi_B\psi_l)\|^2_{L^2(\mathbb{R}^n)} = \int_{\mathbb{R}^n} |\xi|^2|\widehat{\phi_B\psi_l}|^2 d\xi\\
     &\leq  C  \int_{|\xi|<4|\zeta_l|} |\zeta_l|^2|\widehat{\phi_B\psi_l}|^2 d\xi+  C\int_{|\xi|\geq 4|\zeta_l|} |p_{\zeta_l}(\xi)||\widehat{\phi_B\psi_l}|^2 d\xi\\
     &\leq   C|\zeta_l|^2 \|\phi_B\psi_l\|^2_{L^2(\mathbb{R}^n)} + \|\phi_B\psi_l\|^2_{X^{1/2}_{\zeta_l}}\\
     &\leq   C(|\zeta_l|+1) \|\psi_l\|^2_{\dot{X}^{1/2}_{\zeta_l}}
\end{align*}
from Lemma 2.2 in \cite{HT}.
Hence we get
    \begin{align*}
         \|\nabla u_l\|^2_{L^2(B)}
         &\leq \|\zeta_l e^{\zeta_l\cdot x}\gamma_l^{-1/2}(1+\psi_l)\|_{L^2(B)}^2+\| e^{\zeta_l\cdot x}\nabla\gamma_l^{-1/2}(1+\psi_l)\|_{L^2(B)}^2\\
         &\quad+\| e^{\zeta_l\cdot x}\gamma_l^{-1/2}(\nabla \psi_l)\|_{L^2(B)}^2\\
         &\leq  C e^{R|\zeta_l|}\left(1+\|\psi_l\|^2_{\dot{X}^{1/2}_{\zeta_l}}\right).
    \end{align*}
By direct computation we deduce that $\|\psi_l\|_{\dot X^{1/2}_{\zeta_l}}\leq C \|Q_l\|_{\dot X^{-1/2}_{\zeta_l}}\leq C|\zeta_l|^{1/2}$. Then the estimate (\ref{uH1}) holds. Recall that $|\zeta_l|=\sqrt{2}\tau$ .
Summing up, we have
\begin{align*}
              |\langle  Q_2-Q_1| e^{-ir\eta\cdot x} \rangle|
        &\leq   C \langle r\rangle^{1/2} \sum^2_{l,j=1}  \|Q_j\|_{\dot X^{-1/2}_{\zeta_j}}\| \psi_l\|_{\dot X^{1/2}_{\zeta_l}}\\
        &\quad+C\left(\tau^{-1}+\tau^{-\varepsilon/(1+\varepsilon)} +k^2\tau^{-1}\right) \langle r\rangle^{1/2}\| \psi_1\|_{\dot X^{1/2}_{\zeta_1}}\| \psi_2\|_{\dot X^{1/2}_{\zeta_2}}\\
        &\quad +CA^{1/2} e^{\sqrt{2}R\tau}.
\end{align*}
Integrating on both sides and using H\"older inequality, we have
\begin{align*}
              & |\langle  Q_2-Q_1| e^{-ir\eta\cdot x} \rangle|\\
        &\leq   C \langle r\rangle^{1/2}  \sum^2_{l,j=1} \left(\frac{1}{\lambda}\int_{S^{n-1}}\int_{\lambda}^{2\lambda}  \|Q_l\|^2_{\dot X^{-1/2}_{\zeta_l}}   d\tau d\eta_1\right)^{1/2} \left(\frac{1}{\lambda}\int_{S^{n-1}}\int_{\lambda}^{2\lambda}  \|Q_j\|^2_{\dot X^{-1/2}_{\zeta_j}}   d\tau d\eta_1\right)^{1/2}\\
        &\quad  +C\left(\frac{1}{\lambda}\int_{S^{n-1}}\int_{\lambda}^{2\lambda} \left(\tau^{-1}+\tau^{-\varepsilon/(1+\varepsilon)} +k^2\tau^{-1}\right)^2\langle r\rangle    \|Q_1\|^2_{\dot X^{-1/2}_{\zeta_1}} d\tau d\eta_1\right)^{1/2}  \\
        &\quad \left(\frac{1}{\lambda}\int_{S^{n-1}}\int_{\lambda}^{2\lambda}    \|Q_2\|^2_{\dot X^{-1/2}_{\zeta_2}} d\tau d\eta_1\right)^{1/2}
        +C  A^{1/2} e^{2\sqrt{2}R\lambda} .
\end{align*}
By Lemma \ref{intQ}, we deduce
\begin{align}\label{four}
 |\langle  Q_2-Q_1| e^{-ir\eta\cdot x} \rangle|
        &\leq   C \langle r\rangle^{1/2} \left( \left(1+\langle r\rangle^2\lambda^{-1}\right)\lambda^{-\varepsilon/(1+\varepsilon)} +k^4\lambda^{-1} \right)+ CA^{1/2} e^{2\sqrt{2}R\lambda},
\end{align}
where $\tau\in [\lambda,2\lambda]$ and $\tau>\sqrt{2}^{-1}C_*k^2$ with $k\geq 1$ .

Denote that $\mathcal{F}(f)$ the Fourier transformation of a function $f$.
Then we have the following estimates.
\begin{lemma}
     Let $a_0\geq \sqrt{2}^{-1}C_*$ and $k\geq 1$ where $C_*$ is the constant defined in Theorem \ref{CGOestimate}. Then for $r\geq 0$ and $\eta\in S^{n-1}$ the following statements hold: if $0\leq r\leq a_0k^\alpha$, then
     \begin{align}\label{inequ1}
            |\mathcal{F}(Q_2-Q_1)(r\eta)|
    \leq  C\left(\langle r\rangle k^{\frac{-\alpha\varepsilon}{1+\varepsilon}}+k^{4-\alpha} \right)+CA^{1/2} e^{Ck^\alpha};
     \end{align}
     if $r\geq \sqrt{2}^{-1}C_*k^\alpha$, then
     \begin{align}\label{inequ2}
            |\mathcal{F}(Q_2-Q_1)(r\eta)|
       \leq  C\left(\langle r\rangle^{5/2}r^{-1-\frac{\varepsilon}{1+\varepsilon}}+k^{4}r^{-1/2} \right)+ CA^{1/2} e^{Cr}
     \end{align}
     where $\alpha\geq 2$. Here the constant $C$ depends on $n,\Omega,\varepsilon$ and $M$.
\end{lemma}
\begin{proof}
     We consider $\lambda<\tau<2\lambda$. If $0\leq r\leq a_0k^\alpha$, then we take $\tau=a_0k^\alpha$ in (\ref{four}). On the other hand, we let $\tau=r$ when $r\geq \sqrt{2}^{-1}C_*k^\alpha$. This completes the proof.

\end{proof}

\begin{proposition}
     Let $2s>n+3 $. Then we have
     \begin{align}\label{Q12}
          \|Q_2-Q_1\|^2_{H^{-s}(\mathbb{R}^n)}\leq  Ck^4T^{-m} +C(k^{\frac{-2\alpha\varepsilon}{1+\varepsilon}}+k^{8-2\alpha} )+CAe^{Ck^\alpha}+CAe^{CT},
     \end{align}
     where $C$ depends on $n,s,\varepsilon,\Omega$ and $M$.
\end{proposition}

\begin{proof}
     Let $Q=Q_2-Q_1$. Written in polar coordinates, we have that
     \begin{align*}
     \|Q\|^2_{H^{-s}(\mathbb{R}^n)}
     &\leq C \int^\infty_0 \int_{|\eta|=1} |\mathcal{F}Q(r\eta)|^2 \left(1+r^2\right)^{-s}r^{n-1} d\eta dr\\
     &\leq C\Big(\int_0^{a_0k^\alpha} \int_{|\eta|=1} |\mathcal{F}Q(r\eta)|^2 \left(1+r^2\right)^{-s}r^{n-1} d\eta dr\\
     &\quad+\int^T_{a_0k^\alpha} \int_{|\eta|=1} |\mathcal{F}Q(r\eta)|^2 \left(1+r^2\right)^{-s}r^{n-1} d\eta dr\\
     &\quad+\int_T^\infty \int_{|\eta|=1} |\mathcal{F}Q(r\eta)|^2 \left(1+r^2\right)^{-s}r^{n-1} d\eta dr\Big)\\
    & =:  C(I_1+I_2+I_3),
\end{align*}
where $a_0\geq \sqrt{2}^{-1}C_*$ and $T\geq a_0 k^\alpha$ are parameters which will be chosen later.
We estimate $I_3$ first. We get
\begin{align}\label{I3}
     I_3 &\leq C\int_{|\xi|\geq T} \left(1+|\xi|^2\right)^{-s} |\mathcal{F}Q(\xi)|^2  d\xi \notag\\
         &\leq  C T^{-m} \int_{|\xi|\geq T}  \left(1+|\xi|^2\right)^{-1} |\mathcal{F}Q(\xi)|^2d\xi  \notag\\
         &\leq CT^{-m}\|Q_2-Q_1\|^2_{H^{-1}(\mathbb{R}^n)}.
\end{align}
Denote $m:=2s-2$. To estimate $\|Q_2-Q_1\|^2_{H^{-1}(\mathbb{R}^n)}$. Recall that $Q_j=q_j+(k^2+D_j)\gamma^{-1}_j$.
Observe that
\begin{align*}
   \langle q_j| \phi\rangle=&\int_{\mathbb{R}^n}\nabla\gamma_j^{1/2}\cdot\nabla\left(\gamma_j^{-1/2}\phi\right)dx\\
                         =&\int_{\mathbb{R}^n}\nabla\gamma_j^{1/2}\cdot\nabla\gamma_j^{-1/2}\phi
                         +\gamma_j^{-1/2}\nabla\gamma_j^{1/2}\cdot\nabla\phi dx\\
                         =&\int_{\mathbb{R}^n}\mathcal{F}\left(\nabla\gamma_j^{1/2}\cdot\nabla\gamma_j^{-1/2}\right)\mathcal{F}\phi
                         +\mathcal{F}\left(\gamma_j^{-1/2}\nabla\gamma_j^{1/2}\right)\cdot (i\xi) \mathcal{F}\phi d\xi\\
\end{align*}
for any $\phi\in H^1(\mathbb{R}^n)$.
This implies that
\begin{align*}
\|q_j\|^2_{H^{-1}(\mathbb{R}^n)}
&\leq \int_{\mathbb{R}^n}(1+|\xi|^2)^{-1}\left|\mathcal{F}\left(\nabla\gamma_j^{1/2}\cdot\nabla\gamma_j^{-1/2}\right)
                         +\mathcal{F}\left(\gamma_j^{-1/2}\nabla\gamma_j^{1/2}\right)\cdot (i\xi) \right|^2 d\xi\\
                        & \leq \|\nabla\gamma_j^{1/2}\cdot\nabla\gamma_j^{-1/2}\|_{L^2(\mathbb{R}^n)}+ \|\gamma_j^{-1/2}\nabla\gamma_j^{1/2}\|_{L^2(\mathbb{R}^n)} \leq C,
\end{align*}
where $C$ depends on $n,\Omega$ and $M$.
Thus, if $k\geq 1$,
\begin{align*}
         \|Q_2-Q_1\|^2_{H^{-1}(\mathbb{R}^n)}
    \leq \|q_2-q_1\|^2_{H^{-1}(\mathbb{R}^n)} +\|\left(k^2+D_1\right)\gamma^{-1}_1-\left(k^2+D_2\right)\gamma^{-1}_2\|^2_{H^{-1}(\mathbb{R}^n)}\leq Ck^4.
\end{align*}
Hence,
$$
I_3\leq  Ck^4T^{-m}
$$
with $C$ depends on $M$ and $\Omega$.
To estimate $I_1$, we use (\ref{inequ1}) and $2s>n+3$ so that
\begin{align}\label{I1}
     I_1&\leq \int_0^{a_0k^\alpha} \int_{|\eta|=1} \left(  C(\langle r\rangle^2k^{\frac{-2\alpha\varepsilon}{1+\varepsilon}}+k^{8-2\alpha} )
          +CAe^{Ck^\alpha}    \right)(1+r^2)^{-s}r^{n-1} d\eta dr  \notag\\
        &\leq C(k^{\frac{-2\alpha\varepsilon}{1+\varepsilon}}+k^{8-2\alpha} )
          +CAe^{Ck^\alpha}  .
\end{align}
Since $2s>n+3$, we can deduce that
\begin{align}\label{I20}
   & \int^T_{a_0k^\alpha} \int_{|\eta|=1}  \langle r\rangle^{5}r^{-2-\frac{2\varepsilon}{1+\varepsilon}}
             (1+r^2)^{-s}r^{n-1} d\eta dr  \notag\\
        &\leq Ck^{\frac{-2\alpha\varepsilon}{1+\varepsilon}}\int^T_{a_0k^\alpha} \int_{|\eta|=1} \langle r\rangle^{5}r^{-2}
            (1+r^2)^{-s}r^{n-1} d\eta dr
       \leq Ck^{\frac{-2\alpha\varepsilon}{1+\varepsilon}}
\end{align}
and
\begin{align}\label{I21}
     &\int^T_{a_0k^\alpha} \int_{|\eta|=1}k^{8}r^{-1}
            (1+r^2)^{-s}r^{n-1} d\eta dr  \notag\\
        &\leq \int^T_{a_0k^\alpha} \int_{|\eta|=1}k^{8}r^{-2}
             (1+r^2)^{-s+3/2}r^{n-1} d\eta dr
        \leq Ck^{8-2\alpha}.
\end{align}
Applying estimate (\ref{inequ2}), (\ref{I20}) and (\ref{I21}), we have
\begin{align}\label{I2}
     I_2&\leq \int^T_{a_0k^\alpha} \int_{|\eta|=1} \Big(  C(\langle r\rangle^{3-\frac{2\varepsilon}{1+\varepsilon}}+k^{8}r^{-1} )
             +CAe^{Cr}  \Big)(1+r^2)^{-s}r^{n-1} d\eta dr  \notag\\
        &\leq C\left(k^{\frac{-2\alpha\varepsilon}{1+\varepsilon}}+k^{8-2\alpha} \right)
          +CAe^{CT}
\end{align}
for $2s>n+3$.
This completes the proof.

\end{proof}

Now we prove our main theorem.\\
\emph{Proof of Theorem \ref{main}}.

We consider the following two cases:
$$
     (\text{i})\ p\log\frac{1}{A}\geq a_0k^\alpha \ \ \ \hbox{and}\ \ (\text{ii})\ p\log\frac{1}{A}\leq a_0k^\alpha,
$$
where $a_0\geq \sqrt{2}^{-1}C_*$ and $p>0$ are constants which will be determined later. We begin with the case (i). We choose $T=p\log\frac{1}{A}$, which is greater than or equal to $a_0k^\alpha$ by the condition (i). We want to show that there exists $C_1>0$ such that
\begin{align}\label{T1}
     T^{-m}\leq C_1\left(k^\alpha+\log\frac{1}{A} \right)^{-m}
\end{align}
and
\begin{align}\label{T2}
     e^{CT}A\leq C_1\left(k^\alpha+\log\frac{1}{A} \right)^{-m}.
\end{align}
Note that (\ref{T1}) is equivalent to
\begin{align}\label{C1}
    C_1^{-1/m}\left(k^\alpha+\log\frac{1}{A}\right)\leq p\log\frac{1}{A}.
\end{align}
Since
$$
    k^\alpha+\log\frac{1}{A}\leq \left(p a_0^{-1}+1\right) \log\frac{1}{A}
$$
by (i), (\ref{C1}) holds whenever
$$
    C_1^{-1/m}\leq \frac{p}{p a_0^{-1}+1}.
$$
On the other hand, (\ref{T2}) is equivalent to
$$
     CT+\log A\leq \log C_1-m\log\left(k^\alpha+\log\frac{1}{A}\right).
$$
Then
\begin{align}\label{C}
   Cp\log \frac{1}{A}\leq \log \frac{1}{A}+\log C_1-m\log\left(k^\alpha+\log\frac{1}{A}\right).
\end{align}
Note that $\log\left(k^\alpha+\log\frac{1}{A}\right) \leq \log\left(\left(pa_0^{-1}+1\right)\log\frac{1}{A}\right)=\log\left(pa_0^{-1}+1\right)+\log\log\frac{1}{A}$. Assume that $\log\frac{1}{A}\geq 1$. If
\begin{align*}
    Cp\log \frac{1}{A}
    \leq \log C_1+\log\frac{1}{A}-m\left(\log\left(pa_0^{-1}+1\right)+\log\log\frac{1}{A}\right),
\end{align*}
is true, then (\ref{C}) holds.
Choosing
$$
 p\leq \frac{1}{2C},
$$
it is sufficient to show that
$$
  0\leq \log \frac{1}{A}+2\log C_1-2m\log\log\frac{1}{A}
   -2m\log\left(pa_0^{-1}+1\right).
$$
Since
\begin{align*}
     \inf_{0\leq A\leq e^{-1}}\left(\log \frac{1}{A}-2m\log\log\frac{1}{A}\right)
     &=\inf_{z\geq 1}(z-2m\log z)\\
     &\geq  \inf_{z\geq 0}(z-2m\log z)\\
     &=2m-2m\log(2m)\\
     &=\log\left(\frac{e}{2m}\right)\cdot (2m),
\end{align*}
we take
$$
     C_1\geq \left(\frac{p}{a_0}+1\right)^{m}\left(\frac{e}{2m}\right)^{-m}.
$$

Now we consider the case (ii). By choosing $T=a_0k^\alpha$, then from (\ref{Q12}) we have that
$$
\|Q\|^2_{H^{-s}(\mathbb{R}^n)}\leq Ca_0^{-m}k^4 k^{-\alpha m} +C\left(k^{\frac{-2\alpha\varepsilon}{1+\varepsilon}}+k^{8-2\alpha} \right)+CAe^{Ca_0k^\alpha}.
$$
Hence, it remains to show that
\begin{align}\label{k}
     k^{-\alpha m}\leq C_2\left(k^\alpha+\log\frac{1}{A}\right)^{-m},
\end{align}
i.e.,
$$
     k^{\alpha}\geq C_2^{-1/m}\left(k^\alpha+\log\frac{1}{A}\right).
$$
Since
$$
     k^\alpha+\log\frac{1}{A}\leq \left(1+\frac{a_0}{p}\right)k^\alpha
$$
by condition (ii), we have (\ref{k}) if we take $C_2$ large enough so that
$$
     C_2\geq \left(1+\frac{a_0}{p}\right)^m.
$$
Summing up, (\ref{Q12}) can be bounded by
     \begin{align}
          \|Q_2-Q_1\|_{H^{-s}(\mathbb{R}^n)}\leq  C\left(e^{Ck^\alpha}A^{1/2}+k^{\frac{-\alpha\varepsilon}{1+\varepsilon}}+k^{4-\alpha}+ k^2\left(k^\alpha+\log \frac{1}{A}\right)^{-m/2}\right).
     \end{align}
By the definition of $Q_j=q_j+(k^2+D_j)\gamma_j^{-1}$,
we deduce that
\begin{align*}
    k^2\|\gamma_1^{-1}-\gamma_2^{-1}\|_{H^{-s}(\Omega)}&\leq \|D_1\gamma_1^{-1}-D_2\gamma_2^{-1}\|_{H^{-s}(\Omega)}+ \|q_2-q_1\|_{H^{-s}(\Omega)}\\
    &\quad+C\left(e^{Ck^\alpha}A^{1/2}+k^{\frac{-\alpha\varepsilon}{1+\varepsilon}}+k^{4-\alpha}+ k^2\left(k^\alpha+\log \frac{1}{A}\right)^{-m/2}\right).
\end{align*}
Since $\|\gamma_j\|_{C^{1,\varepsilon}(\mathbb{R}^n)}\leq M, \gamma_j >1/(2M)$ and $\|D\|_{L^\infty(\mathbb{R}^n)}\leq M$, we have
$$
     \|D_1\gamma_1^{-1}-D_2\gamma_2^{-1}\|_{H^{-s}(\Omega)}\leq C
$$
and
$$
\|q_2-q_1\|_{H^{-s}(\Omega)}\leq C
$$
with $C$ depends on $n,\Omega$ and $M$.
Hence,
we have
     \begin{align*}
          \|\gamma_1^{-1}-\gamma_2^{-1}\|_{H^{-s}(\Omega)} \leq  Ck^{-2} \left( 1
          +   e^{Ck^\alpha}A^{1/2}+ k^{\frac{-\alpha\varepsilon}{1+\varepsilon}}
          +k^{4-\alpha}+k^2\left(k^\alpha+\log \frac{1}{A}\right)^{-m/2} \right)
     \end{align*}
with $\alpha>2$.
Moreover, for $2s>n+3, n\geq 3$ and $-\log A\geq 1$, we get
\begin{align*}
\left(k^\alpha+\log \frac{1}{A}\right)^{-m/2}\leq k^{\alpha(1-s)}\leq k^{-2}.
\end{align*}
Therefore, the estimate
  \begin{align*}
          \|\gamma_1^{-1}-\gamma_2^{-1}\|_{H^{-s}(\Omega)}
          \leq  Ce^{Ck^\alpha}A^{1/2}+Ck^{-2}+Ck^{2-\alpha},\ \alpha>2
  \end{align*}
holds for some constant $C$ depends on $n, s, \varepsilon, \Omega$ and $M$.
The proof is completed.

\bigskip

\textbf{Acknowlegments.}
The author would like to thank professor Gunther Uhlmann for his encouragements and helpful discussions. The author is partially supported by NSF.

\end{document}